\documentclass{article}
\usepackage{amsmath, amssymb, amsthm, graphicx}
\newtheorem{prop}{Proposition}
\newtheorem{thm}{Theorem}
\newtheorem{defn}{Definition}
\newtheorem{cor}{Corollary}
\title{On normed semigroups}
\author{{\bf V. N. Krishnachandran} \\
Vidya Academy of Science \& Technology\\
(University of Calicut) \\
Thrissur - 680501, Kerala, India\\
{(email: {\tt krishnachandran.v.n@vidyaacademy.ac.in})}
}
\date{}
\begin{document}
\maketitle
\begin{abstract}
The paper begins by exploring the various definitions of norms on semigroups and then presents a new definition of a normed semigroup. The properties of normed semigroups in the new sense are investigated. The new definition of the norm is used to establish a general result on topological regular semigroups which is then used to prove the surprising result that the  semigroup $M_n(\mathbb K)$ is {\em not} a topological regular semigroup. 
\end{abstract}
\section{Introduction}
It appears that there is no one generally accepted definition of a normed
 semigroup. Each author has formulated a definition to suit the immediate
 purposes of his/her requirements. This has resulted in the existence of
 several definitions of normed semigroups. A Google search of the
 literature has brought out  at least six different definitions of norms
 on semigroups. In nearly all of these approaches the semigroups are assumed to have many added structures. In the present paper an attempt is made to define a norm on any arbitrary semigroup without assuming any other additional structure. This is purely a semigroup-theoretic approach. 
Some consequences of the definition are also explored. As an application of the new approach, the concept is used to derive an interesting property of topological regular semigroups. This property is then used to prove the surprising result that the semigroup $M_n(\mathbb K)$ of all $n\times n$ matrices over $\mathbb K=\mathbb R$ or $\mathbb K=\mathbb C$ with the usual topology is not a topological regular semigroup in the sense of Rajan (see \cite{Rajan}). This conclusion makes a revision of the definition toplogical regular semigroup very imperative.
\section{Normed semigroups in the literature}
In the literature, one can see several definitions of normed semigroups.
\begin{enumerate}
\item
%http://gdz.sub.uni-goettingen.de/dms/load/img/?PPN=PPN362162050_0084&DMDID=DMDLOG_0009&LOGID=LOG_0009&PHYSID=PHYS_0055
H. Wegmann (see \cite{Wegmann, Porubsky}) in 1966 introduced the idea of a norm for certain free abelian semigrups as follows: ``A free abelian semigroup $G=\{m,n,\ldots\}$ with identity and a countable system of generators 
$$
E=\{p_1, p_2, p_3,\ldots\}
$$ 
will be called arithmetical if to every element $n$ in $G$ a real-valued norm $|n|$ is assigned such that 
\begin{enumerate}
\item
$|n\cdot m|=|n|\times |m|$ for all $m,n\in G$,
\item
$|p_i|>1$ for all $i$,
\item
$\lim_{i\rightarrow \infty} |p_i|=\infty$.''
\end{enumerate}
\item
In a paper published in 1985, Z. Kryzius defined a normed semigroup follows (see \cite{Kryzius}):
``Let $S$ be a free commutative multiplicative semigrop with identity $e$ and a countable system of generating elements. On it there is defined a positive function $N(a)$. $N(a)$ is completely multiplicative, that is, for all $a,b\in S$ one has $N(ab)=N(a)N(b)$, and $N(e)=1$. Also for any $y\in \mathbb R$ the set $\{a\in S:N(a)\le y\}$ is finite. This function is called the norm."
\item
In \cite{Dikran}, a normed semigroup is defined as follows:
``Let $(S, \cdot)$ be a semigroup. A norm on $S$ is a map $v : S \rightarrow \mathbb R_{\ge 0}$  such that
$$
v(x\cdot y) \le  v(x) + v(y)\text{ for every }  x, y \in S.
$$
A normed semigroup is a semigroup provided with a norm.
If $S$ is a monoid, a monoid norm on $S$ is a semigroup norm $v$ such that
$v(1) = 0$; in such a case $S$ is called normed monoid.''
\item
A normed semigroup is also defined as a normed module over $\mathbb C$ (see \cite{Pavlov}). ``Let $N$ be an abelian semigroup and let $\mathbb C \times N \rightarrow N$, $(\lambda , a)\mapsto \lambda a$, be a given action satisfying the following
conditions: For all $\lambda, \mu \in \mathbb C$ and $a,b\in N$, 
\begin{gather*}
(\lambda \mu)a=\lambda(\mu a)\\
\lambda(a+b) =\lambda a + \lambda b\\
1a =a \\
0a=0
\end{gather*}
Then $N$ is called a module over $\mathbb C$. 
A norm on $N$ is a mapping
$$
||\cdot||:N\rightarrow [0,\infty)
$$
such that,  for all $\lambda \in \mathbb C$ and $a,b\in N$, we have
\begin{gather*}
||a+b||\le ||a|| +||b||\\
||\lambda a|| = |\lambda|\, ||a||\\
||a||=0\Longleftrightarrow  a=0.
\end{gather*}
A semigroup $N$ equipped with the structure of a $\mathbb C$-module and a norm is called a normed (abelian) semigroup.''
\item
%http://arxiv.org/pdf/1209.0974v1
%
Here is yet another definition of a normed semigroup (see, for example, \cite{Shkarin}):
''Let $A$ be an abelian monoid. A norm on $A$ is a function $|\cdot | : A \rightarrow [0,\infty)$ satisfying the conditions
$|na| = n|a|$ and $|a + b| \le |a| + |b|$ for any positive integer $n$  and $a, b \in A$. An abelian monoid equipped with
a norm is a normed semigroup.''
\item
Oscar Valero defines a normed monoid as follows (see \cite{Valero}): ``A normed monoid is a pair $(X,q)$ where $X$ is a monoid and $q:X\rightarrow \mathbb R^+$ satisfying the following conditions:
\begin{enumerate}
\item
$x = 0$ if and only if there is $−x \in  X$ and $q(x) = q(-x) = 0$.
\item
$q(x + y) \le  q(x) + q(y)$ for all $x, y\in X$.
\item
$ q(x) = 0$ if and only if $x = 0$.
\end{enumerate}
\end{enumerate} 
In the next section we give a new definition of a normed semigroup and give some examples. This new definition is more closely related to the algebraic structure of semigroups. 
\section{Normed semigroup}
Let $S$ be a semigroup. A norm on $S$ is a function $\nu:S\rightarrow \mathbb R_{\ge 0}$ satisfying the following condition: For all $a, b \in S$, we have
$$
\nu(ab)\le \nu(a)\nu(b).
$$
\subsection{Examples}
\begin{enumerate}
\item
For any semigroup $S$, the map $\nu:S\rightarrow \mathbb R_{\ge 0}$ defined by
$\nu(a)=0$ for all $a\in S$ defines a norm in $S$.
\item
For any semigroup $S$, the map $\nu:S\rightarrow \mathbb R_{\ge 0}$ defined by
$\nu(a)=1$ for all $a\in S$ defines a norm in $S$.
\item
Let $S$ be a subsemigroup of the multiplicative semigroup $\mathbb R$ of real numbers then the map $\nu:S\rightarrow \mathbb R_{\ge 0}$ defined by
$\nu(a)=|a|$ for all $a\in S$ defines a norm in $S$.
\item
Let $S$ be a subgroup of the additive group $\mathbb R$ of real numbers. Then the following maps define norms on $S$:
\begin{align*}
\nu(x)& = e^x\\
\nu(x)&=e^{|x|}
\end{align*}
\item
Let $M_n(\mathbb K)$ be the semigroup of all square matrices of order $n$ defined over $\mathbb K=\mathbb R$ or $\mathbb K=\mathbb  C$. Then any sub-multiplicative matrix norm on $M_n(\mathbb K)$ is also a semigroup norm. The submultiplicativity ensures that the matrix norm satisfies the defining condition of a semigroup norm.
\end{enumerate}
\section{A semigroup norm on $M_n(\mathbb K)$}
Let $M_n(\mathbb K)$ be as in Example 5 above. For each integer $k$ such that $0<k\le n$ we define, for any $a\in M_n(\mathbb K)$,  
\begin{equation}\label{Norm1}
\nu_k (a) = {n \choose k}\, \max\big\{|m|: m \text{ is a minor of order $k$ of $a$}\big\}
\end{equation}
As special cases we have, for $a=[a_{ij}]\in M_n(\mathbb K)$:
\begin{itemize}
\item
$\nu_n(a)=\det(a)$ and so $\nu_n(ab)=\nu_n(a)\nu_n(b)$. Also $\nu_n(a)=0$ if and only if $a$ is singular.
\item
$\nu_1(a) = n \max\{|a_{ij}|\}$. In this case, $\nu_1(a)=0$ if and only if $a$ is the zero matrix.
\end{itemize}
\begin{prop}
For each integer $k$ such that $0<k\le n$, the function $\nu_k$ as defined in Eq.\eqref{Norm1} defines a semigroup norm on $M_n(\mathbb K)$.
\end{prop}
\begin{proof}
Clearly we have $\nu_k(a)\ge 0$ for all $k$ and for all $a$. So we need only verify that $\nu_k(ab)\le \nu_k(a)\nu_k(b)$ for all $k$ and for all $a,b\in M_n(\mathbb K)$.

We write
$$
a=[a_{ij}], \quad b=[b_{ij}], \quad ab=c=[c_{ij}].
$$

Let $m$ be a minor of order $k$ of $c$. Then $m$ is the determinant of the product of two matrices of the form
\begin{align*}
\alpha & =
\begin{bmatrix}
a_{i_11} & a_{i_12}& \cdots & a_{i_1n} \\
a_{i_21} & a_{i_22} & \cdots & a_{i_2n}\\
\vdots & & & \\
a_{i_k1} & a_{i_k2} & \cdots & a_{i_kn}
\end{bmatrix}
= [\alpha_{ij}]\quad \text{(say)} \\
\beta & =
\begin{bmatrix}
b_{1j_1} & b_{1j_2} & \cdots & b_{1j_k} \\
b_{2j_1} & b_{2j_2} & \cdots & b_{2j_k} \\
\vdots & & & \\
b_{nj_1} & b_{nj_2} & \cdots & b_{nj_k} 
\end{bmatrix} = [\beta_{ij}]\quad \text{(say)}
\end{align*}
Here $1\le i_1\le i_2 \le \cdots \le i_k\le n$ and $1\le j_1\le j_2\le \cdots \le j_k\le n$.

Using the Cauchy-Binet  formula for the evaluation of the determinant of the product of two (not necessarily square) matrices (see, for example, \S 5.14 \cite{Loehr}, or \S 3.2.2 \cite{Serre}) we have
\begin{align}
m & =\det (\alpha\beta) \notag \\
& =\sum_{(p)}
\begin{vmatrix}
\alpha_{1p_1} & \cdots & \alpha_{1p_k}  \\
\vdots & & \\
\alpha_{kp_1} & \cdots & \alpha_{kp_k}
\end{vmatrix}
\times
\begin{vmatrix}
\beta_{p_11}& \cdots & \beta_{p_1k}\\
\vdots & & \\
\beta_{p_k1} & \cdots & \beta_{p_kk}
\end{vmatrix}\label{minor}
\end{align}
where the summation is over all possible permutations of the indices $1\le p_1 < p_2 < \cdots < p_k\le k$. 

Firstly, we note that there are $n\choose k$ terms in the sum given in Eq.\eqref{minor}. Secondly, by the definition of $\nu_k$, the absolute value of the first factor of each term in Eq.\eqref{minor} is not greater than $\nu_k(a)/{n\choose k}$ and the absolute value of second factor is not greater than $\nu_k(b)/{n \choose k}$. Hence, we have
$$
|m|\le {n\choose k}\times \frac{\nu_k(a)}{{n \choose k}} \times \frac{\nu_k(b)}{{n\choose k}}
$$
and so
$$
{n\choose k} |m|\le \nu_k(a)\nu_k(b).
$$
Since this is true for all minors of $c$ of order $k$, we have
$$
\nu_k(ab)\le \nu_k(a)\nu_k(b).
$$
Thus $\nu_k$ is indeed a semigroup norm on $M_n(\mathbb K)$.
\end{proof}
\section{Properties}
Let $\nu$ be a norm on a semigroup $S$. We denote by $E(S)$ the set of idempotents in $S$. We also use other standard notations of algebraic theory of semigroups.
\begin{prop}
If $e\in E(S)$ then $\nu(e)=0$ or $\nu(e)\ge 1$.
\end{prop}
\begin{proof}
Since $e^2 =e$, by the defining property of a norm, we have
$$
\nu(e)=\nu(e^2)\le \nu(e)\nu(e).
$$
It follows that if $\nu(e)\ne 0$ then we must have $\nu(e)\ge 1$.
\end{proof}
\begin{prop}
The set $Z=\{x\in S: \nu(x)=0\}$ is a subsemigroup of $S$.
\end{prop}
\begin{proof}
Let $a,b\in Z$. Then $\nu(a)=\nu(b)=0$. Hence
$$
\nu(ab)\le \nu(a)\nu(b)=0.
$$
But since $\nu(ab)\ge 0$, we must have $\nu(ab)=0$ and so $ab\in Z$. 
It follows that $Z$ is a subsemigroup of $S$.
\end{proof}
\begin{prop}
If $\nu(a)=0$ then $\nu(b)=0$ for all $b\in D_a$.
\end{prop}
\begin{proof}
Let $\nu(a)=0$ and let $a\mathcal D b$. Then there is $c$ such that $a\mathcal R c \mathcal b$. Hence we can find $x,y\in S^1$ such that 
$$
c=ax, \quad b=yb.
$$
Now we have
$$
0\le \nu(c)=\nu(ax)\le\nu(a)\nu(x)=0
$$
and so $\nu(c)=0$. Similarly, form $b=yc$ it follows that $\nu(b)=0$.
\end{proof}
As a corollary we note that if $\nu(a)\ne 0$ then $\nu(b)\ne 0$ for all $b\in D_a$.
\begin{prop}
If $b$ is an inverse of $a$ and if $\nu(a)\ne 0$ then $\nu(b)\ge 1/\nu(a)$.
\end{prop}
\begin{proof}
If $b$ is an inverse of $a$ we have 
$$
a=aba.
$$
Thus
$$
\nu(a)\le\nu(a)\nu(b)\nu(a).
$$
If $\nu(a)\ne 0$ then, canceling out this nonzero factor we get
$$
1\le \nu(a)\nu(b)
$$
from which the result follows.
\end{proof}
\begin{prop}
Let  $S$ be  a finite group and $\nu(a)\ne 0$ for all $a\in S$, Then $\nu(a)\ge 1$ for all $a\in S$.
\end{prop}
\begin{proof}
Let
$$
\lambda = \min \{ \nu(x): x\in S\}.
$$
Since $S$ is finite, the set $\{\nu(x):x\in S\}$ is also finite. Hence $\lambda\ne 0$ and there must be an element $c\in S$ such that $\lambda = \nu(c)$. 
Now we have 
$$
\nu(c^2)=\nu(c)\nu(c)=\lambda^2.
$$
If $\lambda<1$ then $\lambda^2<\lambda$ and so $\nu(c^2)<\lambda$. This would contradict the fact that $\lambda$ is the minimal element of the set 
$\{\nu(x):x\in S\}$ . Hence we must have $\lambda \ge 1$. It follows that $\nu(x)\ge 1$ for all $x\in S$.
\end{proof}
The conclusion of Proposition 6 need not be true if the assumption that $S$ is finite is dropped. Consider the multiplicative semigroup $S=\mathbb R\setminus \{0\}$ with the norm $\nu(x)=|x|$. In $S$, we have $\nu(a)\ne 0$ for all $a$ but it is not true that $\nu(a)\ge 1$ for all $a\in S$.
\begin{prop}
If there is a a zero element (left, right, or two-sided) $z$ in  a normed semigroup $S$ such that $\nu(z)\ne 0$ then $\nu(x)\ge 1$ for all $x\in S$.
\end{prop}
\begin{proof}
Let $z$ be a left zero in $S$. Then we have $z=zx$ for all $x\in S$. Hence
$$
\nu(z)=\nu(zx)\le \nu(z)\nu(x).
$$
Canceling out the nonzero factor $\nu(z)$ we get $\nu(x)\ge 1$ for all $x$ in $S$. The same argument can be applied to the case where the zero element is a right zero.
\end{proof}
The next proposition shows that if $\nu(b)=0$ for some $b$ in $S$, then all elements in $S$ which are ``less than'' $b$ relative to the natural partial order in $S$ have norm $0$ (see \cite{Mitsch} for more details on natural partial orders on arbitrary semigroups). 
\begin{prop}
Let $a,b\in S$ and $a\le b$ in the natural partial order in $S$. 
If $\nu(b)=0$ then $\nu(a)=0$.
\end{prop}
\begin{proof}
Let $a\le b$ in the natural partial order in $S$ and let $\nu(b)=0$. 
Then, by definition,  there exist $x,y\in S^1$ such that 
$$
a = xb = by, \quad xa = a.
$$
If $x=1$ then $a=b$ and so $\nu(a)=0$. Otherwise we have
$$
\nu(a)=\nu(xb)\le \nu(x)\nu(b) = 0.
$$
Hence the result.
\end{proof}
\section{Topological regular semigroups}
\subsection{Definitions}
Let $X$ and $Y$ be nonempty sets. Then a map $f : X \rightarrow  P(Y )$, where $ P(Y )$ is the
family of subsets of $Y$, is called a multiple map of $X$ into $Y$. We write $f : X \rightarrow  Y$ in
this case also. If $g : A \rightarrow  B$ is a function then 
$g^{-1}: B\rightarrow A$ defined in the obvious
way is a multiple map. If S is a regular semigroup then the map $V : S \rightarrow  S$ defined
by $x\mapsto  V(x)$ , where $V(x)$ is the set of inverses of $x$, is a multiple map.
If $X$ and $Y$ are topological spaces and $f : X \rightarrow  Y$ is a multiple map then we say
that $f$ is continuous if for each open set $B \subseteq  Y$ , the set
$$
f^{-1}(B) =\{ x \in X : f(x) \cap B \ne \emptyset \}
$$
is open in $X$, and $f$ is said to be open if for each open set $A \subseteq  X$ , the set
$$
f(A)= \{ y \in  Y : y\in  f(x) \text{ for some }x \in A\}
$$
is open in $Y$.

In the following, Definition 1 is standard (see \cite{Carruth}) and Definition 2 is due to Rajan (see \cite{Rajan}).
\begin{defn}
A topological semigroup is a triple $(S, \circ, \tau)$ where $(S, \tau)$  is a
topological space and $(S, \circ)$ is a semigroup such that the product map $\circ : S\times S \rightarrow S$ 
is continuous when $S \times S$ is given the product topology.
\end{defn}
\begin{defn}
Let $S$ be a topological semigroup which is also a regular semigroup.
Then $S$ is called a topological regular semigroup  if the multiple map 
$V : S \rightarrow S$ is open.
\end{defn}
\subsection{An interesting property}
The following theorem can be used to check whether a regular semigroup equi-pped with  a continuous norm is a topological regular semigroup. As a corollary to the theorem, we prove the surprising result that sthe semigroup $M_n(\mathbb K)$ equipped with the usual topology is not a topological regular semigroup in the sense of Definition 2.
\begin{thm}
Let $S$ be a regular semigroup which is also a topological semigroup.
Let $\nu$ be a continuous norm on $S$. If the set 
$$
N = \{ x\in  S : \nu(x)\ne 0\} 
$$
is not closed in $S$ then $S$ is not a topological regular semigroup as defined in Definition 2. 
\end{thm}
\begin{proof}
Let $N$ be not closed in $S$. Let $a \in  S\setminus  N$ be in the closure of $N$. Since $a\in S\setminus N$, we have $\nu(a)=0$.

Let $a^\prime$  be an inverse of $a$. Since $a^\prime  = a^\prime aa^\prime$, properties of norm imply  that
$\nu(a^\prime) = 0$ so that $a^\prime \in S\setminus N$. 

Now choose $\epsilon >0$  arbitrarily. By the continuity of $\nu$ we
can find an open neighbourhood $G^\prime$  of $a^\prime$  such that 
$\nu(x) < \epsilon$  for all $x \in G^\prime$. We shall
show that $V(G^\prime)$ is not open $S$.

Obviously $ a\in V(G^\prime)$. Let $H^\prime$ be an arbitrary neighbourhood of $a$. 
By the continuity of $\nu$ we can find an open neighbourhood $H^{\prime\prime}$ of $a$ such that $\nu(y)< \frac{1}{2\epsilon}$ for every $y\in H^{\prime\prime}$. 
\begin{center}
\includegraphics{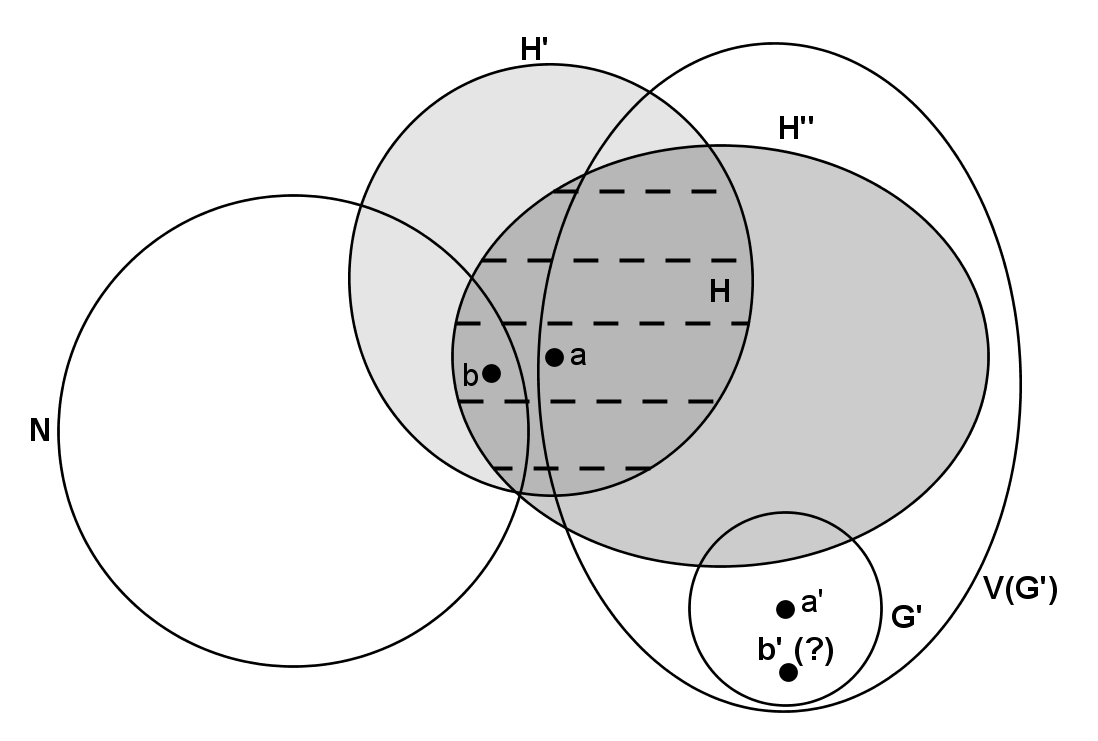}
\end{center}
Now let $H=H^\prime \cap H^{\prime\prime}$. $H$ is an open neighbourhood of $a$ which is contained in $H^\prime$. Since $a$ is in the closure of $N$, there must be an element $b\in H$ such $\nu(b)\ne 0$. We show below that $b\not\in V(G^\prime)$. This would then imply that $H\not \subseteq V(G^\prime)$ and so $H^\prime \not \subseteq V(G^\prime)$. Thus any neighbourhood of $a$ would not be completely contained in $V(G^\prime)$. It would follow that $a$ is not an interior point of $V(G^\prime)$ and so $V(G^\prime)$ would not be an open set.

To show that $b\not\in V(G^\prime)$, assume the contrary and let $b$ be an inverse of $b^\prime \in  G^\prime$. By the choice of $G^\prime$  we have $\nu(b^\prime)<\epsilon$. Since $\nu(b)\ne 0$  and $b = bb^\prime b$,
we have
$$
\nu(b^\prime) \ge \frac{1}{\nu(b)} >\frac{1}{1/2\epsilon}=2\epsilon
$$
which is a contradiction. Thus we must have $b\not\in V(G^\prime)$. 

Since $G^\prime$ is open and $V(G^\prime)$ is not, $S$ cannot be a topological regular semigroup.
\end{proof}
\begin{cor}
The semigroup $M_n(\mathbb K)$ of all square matrices of order $n$ equipped with the usual topology is not a topological regular semigroup.
\end{cor}
\begin{proof}
Let $0<k <n$ and let $\nu_k$ be defined as in Eq.\eqref{Norm1}. Then we have already proved that s$\nu_k$ is a norm on $M_n(\mathbb K)$. Since the determinant and max functions are continuous, the function $\nu_k$ is also continuous. Now we have
\begin{align*}
N_k & = \{ x\in M_n(\mathbb K):\nu_k(x)\ne 0\}\\
& = \{ x\in M_n(\mathbb K) : \text{rank }(x)\ge k\}.
\end{align*}
For each positive integer $m$, let
$$
x_m=\begin{bmatrix} \frac{1}{m}I & O \\ O & O \end{bmatrix}
$$
where $I$ is the identity matrix of order $k$ and $O$'s are zero-matrices. 
Then $x_m\in N_k$ and 
$$
\lim_{m\rightarrow \infty}x_m=0, \text{ the zero-matrix in $M_n(\mathbb K)$}.
$$
But the zero-matrix is not in $N_k$. It follows that $N_k$ is not closed in $M_n(\mathbb K)$. Therefore by theorem $M_n(\mathbb K)$ is not a topological regular semigroup.
\end{proof}
\paragraph{Remarks.} The conclusion of Corollary 1 indicates the inadequecy of Definition 1 as a `proper' definition of a topological regular semigroup. Since $M_n(\mathbb K)$ is a natural example of a topological regular semigroup, any definition of a topological regular semigroup which excludes $M_n(\mathbb K)$ as an example should be treated as unsatisfactory. An attempt to reformulate the definition of a topological regular semigroup which addresses this issue has been made in \cite{Krish}.

\end{document}